\documentclass[12pt]{article}
\usepackage{amssymb,amsmath}
\usepackage{amsthm}
\usepackage{cases}
\usepackage{amsfonts}
\usepackage{graphicx}
\usepackage{cite,color,xcolor}
\usepackage[left=2.6cm,right=2.6cm,top=2.4cm,bottom=2.4cm]{geometry}
\usepackage[colorlinks,citecolor=blue,urlcolor=blue]{hyperref}
\usepackage[utf8]{inputenc}

\newtheorem{theorem}{Theorem}[section]

\newtheorem{lemma}{Lemma}[section]
\newtheorem{proposition}{Proposition}[section]

\newtheorem{definition}{Definition}[section]
\newtheorem{remark}{Remark}[section]

\usepackage{txfonts}
\newcommand{\bal}{\begin{align}}
\newcommand{\bbal}{\begin{align*}}
\newcommand{\beq}{\begin{equation}}
\newcommand{\eeq}{\end{equation}}
\newcommand{\bca}{\begin{cases}}
\newcommand{\eca}{\end{cases}}
\def\div{\mathord{{\rm div}}}
\newcommand{\pa}{\partial}
\newcommand{\fr}{\frac}
\newcommand{\na}{\nabla}
\newcommand{\De}{\Delta}

\newcommand{\cd}{\cdot}
\newcommand{\ep}{\varepsilon}
\newcommand{\dd}{\mathrm{d}}

\newcommand{\R}{\mathbb{R}}

\newcommand{\les}{\lesssim}

\newcommand{\f}{\left}
\newcommand{\g}{\right}
\newcommand{\Div}{\mathrm{div\,}}
\newcommand{\pr}{\\ &}

\begin{document}
\bibliographystyle{plain}
\title{Loss of Uniform Convergence for Solutions of the Navier--Stokes Equations in the Inviscid Limit}

\author{Jinlu Li$^{1}$, Yanghai Yu$^{2,}$\footnote{E-mail: lijinlu@gnnu.edu.cn; yuyanghai214@sina.com(Corresponding author); mathzwp2010@163.com} and Weipeng Zhu$^{3}$\\
\small $^1$ School of Mathematics and Computer Sciences, Gannan Normal University, Ganzhou 341000, China\\
\small $^2$ School of Mathematics and Statistics, Anhui Normal University, Wuhu 241002, China\\
\small $^3$ School of Mathematics and Big Data, Foshan University, Foshan, Guangdong 528000, China}

\date{\today}
\maketitle\noindent{\hrulefill}

{\bf Abstract:} In this paper, we consider the inviscid limit problem to the higher dimensional incompressible Navier--Stokes equations in the whole space. It is shown in \cite{GZ} that given initial data $u_0\in B^{s}_{p,r}$ and for some $T>0$, the solutions of the Navier--Stokes equations converge strongly in $L^\infty_TB^{s}_{p,r}$ to the Euler equations as the viscosity parameter tends to zero. We furthermore prove the failure of the uniform (with respect to the initial data) $B^{s}_{p,r}$ convergence in the inviscid limit of a family of solutions of the Navier–Stokes equations towards a solution of the Euler equations.


\vskip0mm\noindent{\hrulefill}

\section{Introduction}
In this article, we consider the Cauchy problem for the incompressible Navier--Stokes equations
\begin{align*}
\rm{(NS)}\qquad\begin{cases}
\pa_t u+u\cdot \nabla u-\ep \Delta u+\nabla P=0,\\
\mathrm{div\,} u=0,\\
u(0,x)=u_0(x),
\end{cases}
\end{align*}
where the vector field $u(t,x):[0,\infty)\times {\mathbb R}^d\to {\mathbb R}^d$ stands for the velocity of the fluid, the quantity $P(t,x):[0,\infty)\times {\mathbb R}^d\to {\mathbb R}$ denotes
the scalar pressure, and $\mathrm{div\,} u=0$ means that the fluid is incompressible.
When the viscocity vanishes ($\ep=0$), then (NS) reduces to the Euler equations for ideal incompressible fluid
\begin{align*}
\rm{(Euler)}\qquad\begin{cases}
\pa_t u+u\cdot \nabla u+\nabla P=0, \\
\mathrm{div\,} u=0,\\
u(0,x)=u_0(x).
\end{cases}
\end{align*}
The mathematical study of both Navier--Stokes and Euler equations has a long and distinguished history (see Constantin's survey \cite{P} for more details) and the problems of global regularity for 3D equations are still challenging open problems. We do not detail the literature since it is huge and refer the readers to see the monographs of Majda-Bertozzi \cite{Majda} and Bahouri-Chemin-Danchin \cite{B.C.D} for the well-posedness results of both Navier--Stokes and Euler equations in Sobolev and Besov spaces respectively.

A classical problem in fluid mechanics is the approximation in the limit $\ep\to0$ of vanishing viscosity (also called inviscid limit) of solutions of the Euler equations by solutions of the incompressible Navier–Stokes equations.
The problem of the convergence of smooth viscous solutions of (NS) to the Eulerian one as $\ep\to0$ is
well understood and has been studied in many literatures. See for example \cite{Swann, Kato}, and \cite{CKV} for the inviscid limit on the bounded domain. Majda \cite{Majda} showed that under the assumption $u_0\in H^s$ with $s>\frac{d}{2}+2$, the solutions $u_\ep$ to (NS) converge in $L^2$ norm as $\ep\to 0$ to the unique solution of Euler equation and the convergence rate is of order $(\ep t)^{\fr12}$. Masmoudi \cite{M} improved the result and proved the convergence in $H^s$ norm under the weaker assumption $u_0\in H^s$ with $s>\frac{d}{2}+1$. In dimension two the results are global in time and were improved in \cite{HK} where the assumption is improved to $u_0\in B^2_{2,1}$ with convergence in $L^2$, and further generalized to other Besov spaces $B^{2/p+1}_{p,1}$ with convergence in $L^p$, see section 3.4 in \cite{MWZ}. In three dimension a similar result was proved in \cite{Wu} for axis-symmetric flows without swirl. Chemin \cite{chem} resolved inviscid limit of Yudovich type solutions with only the assumption that the vorticity is bounded. In the case of two dimensional (torus or the whole space) and rough initial data, by taking greater advantage of vorticity formulation, more beautiful results are obtained quantitatively (see for example \cite{Be,Ci,P2}). Recently, Bourgain-Li in \cite{B1,B2} employed a combination of Lagrangian and Eulerian techniques to obtain strong local ill-posedness results in borderline Besov spaces $B^{d/p+1}_{p,r}$ for $1\leq p<\infty$ and $1<r\leq\infty$ when $d=2,3$. Guo-Li-Yin \cite{GZ} solved the inviscid limit in the same topology.
For any $R>0$, from now on, we denote any bounded subset $U_R$ in $B^s_{p,r}(\mathbb{R}^d)$ by
$$U_R:=\left\{\phi\in B^s_{p,r}(\mathbb{R}^d): \|\phi\|_{B^s_{p,r}(\mathbb{R}^d)}\leq R,\;\mathrm{div\,} \phi=0\right\}.$$
In \cite{GZ}, Guo-Li-Yin obtained
\begin{theorem}\label{th1} Let $d\geq 2$ and $\ep\in [0,1]$. Assume that $(s,p,r)$ satisfies
\begin{align}\label{con1}
s>\frac{d}{p}+1,\; (p,r)\in [1,\infty]\times [1,\infty) \quad \mbox{   or   } \quad
s=\frac{d}{p}+1, \;(p,r)\in [1,\infty]\times\{1\}.
\end{align}
Given initial data $u_0\in U_R$, there exists $T=T(R,s,p,r,d)>0$ such that (NS) has a unique solution $\mathbf{S}_{t}^\ep(u_0)\in C([0,T];B^s_{p,r})$. Moreover, there holds

(1)\; Uniform bounds: there exists $C=C(R,s,p,r,d)>0$ such that
\begin{align*}
\f\|\mathbf{S}_{t}^\ep(u_0)\g\|_{L_T^\infty B^s_{p,r}}\leq C, \quad \forall\ \ep\in [0,1].
\end{align*}
Moreover, if $u_0\in B^{\gamma}_{p,r}$ with $\gamma>s$, then $\exists\ C_2=C_2(R, \gamma,s,p,r,d)>0$
\begin{align*}
\f\|\mathbf{S}_{t}^\ep(u_0)\g\|_{L_T^\infty B^\gamma_{p,r}}\leq  C_2\|u_0\|_{B^\gamma_{p,r}}.
\end{align*}

(2) Inviscid limit: given $u_0\in U_R$, $\forall \eta>0$, $\exists\ \ep_0=\ep_0(u_0,\eta,T)$, $\forall \ep\leq\ep_0$: $\f\|\mathbf{S}_{t}^\ep(u_0)-\mathbf{S}_{t}^0(u_0)\g\|_{L_T^\infty B^s_{p,r}}\leq \eta$, namely,
\begin{align}\label{inv}
\lim_{\ep\to 0}\f\|\mathbf{S}_{t}^\ep(u_0)-\mathbf{S}_{t}^0(u_0)\g\|_{L_T^\infty B^s_{p,r}}=0.
\end{align}
\end{theorem}
We should emphasize that the statement \eqref{inv} is a consequence of the strong convergence of the solutions to (NS) with the initial data in the bounded set $U_R$ of $B^s_{p,r}$, i.e., $\mathbf{S}_{t}^\ep(u_0)\to\mathbf{S}_{t}^0(u_0)$ in $L_T^\infty B^s_{p,r}$ for each $u_0\in U_R$.
An interesting problem appears: whether the convergence can be made uniformly with respect to the initial data?
In other words, for any $u_0\in U_R$, is the statement \eqref{inv} valid? We will give a negative answer. Our main result is the following.

\begin{theorem}\label{th2} Let $d\geq 2$ and $\ep\in [0,1]$. Assume that $(s,p,r)$ satisfies \eqref{con1}. For any $u_0\in U_R$, let $\mathbf{S}_{t}^{\ep_n}(u_0)$ and $\mathbf{S}_{t}^{0}(u_0)$ be the solutions of (NS) and (Euler) with the same initial data $u_0$, respectively. Then a family of solutions of the Navier–Stokes equations $\f\{\mathbf{S}_{t}^{\ep}(u_0)\g\}_{\ep>0}$
\begin{equation*}
\mathbf{S}_t^\ep:\begin{cases}
U_R \rightarrow \mathcal{C}([0, T] ; B_{p, r}^{s}),\\
u_0\mapsto \mathbf{S}_t^\ep(u_0),
\end{cases}
\end{equation*}
do not converge strongly in a uniform way with respect to initial data to the solution $\mathbf{S}_{t}^{0}(u_0)$ of (Euler) in $B^{s}_{p,r}$.
Moreover, there exists a sequence initial data $u^n_0\in U_R$, $\exists\ \eta_0>0$, $\forall \ep>0$, $\exists\ \ep_n\leq \ep$ such that for a short time $T_0\leq T$
$$
\liminf_{n\to \infty}\left\|\mathbf{S}_{t}^{\ep_n}(u^n_0)-\mathbf{S}_{t}^{0}(u^n_0)\right\|_{L^\infty_{T_0}B^s_{p,r}}\geq \eta_0.
$$
\end{theorem}

For any $R>0$, we denote the $B_{p, r}^{s}$-neighbourhood of $u_0$ by
$$\mathcal{N}_{R}(u_0)=\f\{\phi\in B_{p, r}^{s}:\;\|\phi-u_0\|_{B_{p, r}^{s}}\leq R,\,\mathrm{div\,} \phi=\mathrm{div\,} u_0=0\g\}.$$

We have the following stronger result.
\begin{theorem}\label{th3} Let $d\geq 2$ and $\ep\in [0,1]$. Assume that $(s,p,r)$ satisfies \eqref{con1}. Given that $\psi\in B_{p, r}^{s}$, then a family of solutions of the Navier–Stokes equations $\f\{\mathbf{S}_{t}^{\ep}(\psi)\g\}_{\ep>0}$
\begin{equation*}
\mathbf{S}_t^\ep:\begin{cases}
\mathcal{N}_{R}(\psi) \rightarrow \mathcal{C}([0, T] ; B_{p, r}^{s}),\\
\psi\mapsto \mathbf{S}_t^\ep(\psi),
\end{cases}
\end{equation*}
do not converge strongly in a uniform way with respect to initial data to the solution $\mathbf{S}_{t}^{0}(\psi)$ of (Euler) in $B^{s}_{p,r}$.
Moreover, given that $\psi\in B_{p, r}^{s}$, there exists a sequence initial data $u^n_0\in U_R$, $\exists\ \eta_0>0$, $\forall \ep>0$ , $\exists\ \ep_n\leq \ep$ such that for a short time $T_0\leq T$
$$
\liminf_{n\to \infty}\left\|\mathbf{S}_{t}^{\ep_n}(\psi+u^n_0)-\mathbf{S}_{t}^{0}(\psi+u^n_0)\right\|_{L^\infty_{T_0}B^s_{p,r}}\geq \eta_0.
$$
\end{theorem}
\begin{remark}\label{re1}
To the best of our knowledge, Theorem \ref{th2} is the first result even in the Sobolev topology on the non-uniform convergence of the solutions for the Navier--Stokes Equations in the inviscid limit.
\end{remark}
\begin{remark}\label{re2}
Theorem \ref{th3} implies that, for any $\psi\in B^s_{p,r}$, the solutions of the Navier–Stokes equations $\f\{\mathbf{S}_{t}^{\ep}(u_0)\g\}_{\ep>0}$ can not converge uniformly to the Eulerian one as $\ep\to0$.
\end{remark}

Here let us give an overview of the strategy.

\noindent\textbf{Strategies to the proof of Theorem \ref{th2}.}
Due to the incompressible condition, the pressure can be eliminated from (NS) and (Euler). In fact, applying the Leray operator $\mathcal{P}$ to (NS) and (Euler), respectively, then we have
\begin{align*}
\begin{cases}
\pa_t u-\ep \Delta u+\mathcal{P}(u\cdot \nabla u)=0,\\
u(0,x)=u_0(x),
\end{cases}
\quad\text{and}\qquad
\begin{cases}
\pa_t u+\mathcal{P}(u\cdot \nabla u)=0, \\
u(0,x)=u_0(x).
\end{cases}
\end{align*}
(1) We decompose the solution $\mathbf{S}^{\ep}_t(u_0)$ to (NS) into three parts
\begin{align}\label{U}
\mathbf{S}^{\ep}_t(u_0)=u_1+u_2+ \text{\bf Error-1} ,
\end{align}
where $u_1$ and $u_2$ solves the following system, respectively,
\begin{align*}
\begin{cases}
\pa_t u_1-\ep\Delta u_1=0, \\
u_1(0,x)=u_0(x),
\end{cases}
\quad\text{and}\qquad
\begin{cases}
\pa_t u_2-\ep\Delta u_2+\mathcal{P}\f(u_1\cdot \nabla u_1\g)=0, \\
u_2(0,x)=0.
\end{cases}
\end{align*}
(2) We also decompose the solution $\mathbf{S}^{0}_t(u_0)$ to (Euler) into three parts
\begin{align}\label{V}
\mathbf{S}^{0}_t(u_0)=v_1+v_2+ \text{\bf Error-2},
\end{align}
where
\begin{align*}
v_1=u_0\quad\text{
and}\quad v_2=\int_0^t\mathcal{P}\f(v_1\cdot \nabla v_1\g)\dd\tau=t\mathcal{P}\f(u_0\cdot \nabla u_0\g).
\end{align*}
To ensure that the difference between the solutions $\mathbf{S}^{\ep}_t(u_0)$ to (NS) and $\mathbf{S}^0_t(u_0)$ to (Euler) in the $B^s_{p,r}$-topology is bounded below by a positive constant at any later time, we need to construct the initial data sequence $u_0\in U_R$ to satisfy that the following three conditions: $\forall t\in(0,1]$ and $i=1,2$
\begin{itemize}
  \item $\f\|u_1-v_1\g\|_{B^s_{p,r}}\approx t, \qquad \text{see Proposition}\; \ref{pr1}$;
  \item $\f\|u_2-v_2\g\|_{B^s_{p,r}}\lesssim t^2, \qquad \text{see  Proposition}\; \ref{pr2}$;
  \item $\f\|\text{\bf Error-i}\g\|_{B^s_{p,r}}\lesssim t^2, \qquad \text{see  Proposition}\; \ref{pr3}$;
\end{itemize}
\noindent\textbf{Strategies to the proof of Theorem \ref{th3}.} We decompose the difference of $\mathbf{S}^{\ep_n}_{t}(\psi+u^{n}_{0})$ and $\mathbf{S}^{0}_{t}(\psi+u^{n}_{0})$ as follows
\bal\label{mo}
\mathbf{S}^{\ep_n}_{t}(\psi+u^{n}_{0})-\mathbf{S}^{0}_{t}(\psi+u^{n}_{0})
&=\underbrace{\mathbf{S}^{\ep_n}_{t}(u^{n}_{0})-\mathbf{S}^{0}_{t}(u^{n}_{0})}_{\text{Theorem\,} \ref{th2}}+\underbrace{\mathbf{S}^{\ep_n}_{t}(\psi)-\mathbf{S}^{0}_{t}(\psi)}_{\text{Theorem\,} \ref{th1}}\nonumber\\
&\quad+\underbrace{\mathbf{S}^{\ep_n}_{t}(S_n\psi)-\mathbf{S}^{\ep_n}_{t}(\psi)
+\mathbf{S}^{\ep_n}_{t}(\psi+u^{n}_{0})-\mathbf{S}^{\ep_n}_{t}(S_n\psi+u^{n}_{0})}_{\text{Proposition}\,\, \ref{pro5-1}}
\nonumber\\
&\quad+\underbrace{\mathbf{S}^{\ep_n}_{t}(S_n\psi+u^{n}_{0})-\mathbf{S}^{\ep_n}_{t}(S_n\psi)
-\mathbf{S}^{\ep_n}_{t}(u^{n}_{0})}_{\text{Proposition}\,\, \ref{pro5-2}}+\cdots.
\end{align}
We would like to mention that \eqref{mo} is the most difficulty term. To bypass this, we need to take advantage of the decay of solutions at infinity since the solutions is smooth. It promotes us to modify the initial data which also possess the decay at infinity, see Section \ref{sec4} for more details.

\noindent\textbf{Organization of our paper.} In Section \ref{sec2}, we list some notations and known results which will be used in the sequel. We prove Theorem \ref{th2} and Theorem \ref{th3} in Section \ref{sec3} and Section \ref{sec4}, respectively.
\section{Preliminaries}\label{sec2}
We will use the following notations throughout this paper.
\begin{itemize}
  \item For $X$ a Banach space and $I\subset\R$, we denote by $\mathcal{C}(I;X)$ the set of continuous functions on $I$ with values in $X$.
  \item The symbol $\mathrm{A}\lesssim (\gtrsim)\mathrm{B}$ means that there is a uniform positive ``harmless" constant $\mathrm{C}$ independent of $\mathrm{A}$ and $\mathrm{B}$ such that $A\leq(\geq) \mathrm{C}B$, and we sometimes use the notation $\mathrm{A}\approx \mathrm{B}$ means that $\mathrm{A}\lesssim \mathrm{B}$ and $\mathrm{B}\lesssim \mathrm{A}$.
  \item Let us recall that for all $u\in \mathcal{S}'$, the Fourier transform $\mathcal{F}u$, also denoted by $\widehat{u}$, is defined by
$$
\mathcal{F}u(\xi)=\widehat{u}(\xi)=\int_{\R^d}e^{-\mathrm{i}x\cd \xi}u(x)\dd x \quad\text{for any}\; \xi\in\R^d.
$$
  \item The inverse Fourier transform allows us to recover $u$ from $\widehat{u}$:
$$
u(x)=\mathcal{F}^{-1}\widehat{u}(x)=(2\pi)^{-d}\int_{\R^d}e^{\mathrm{i}x\cdot\xi}\widehat{u}(\xi)\dd\xi.
$$
 \item We denote the projection
\bbal
&\mathcal{P}: L^{p}(\mathbb{R}^{d}) \rightarrow L_{\sigma}^{p}(\mathbb{R}^{d}) \equiv \overline{\left\{f \in \mathcal{C}^\infty_{0}(\mathbb{R}^{d}) ; {\rm{div}} f=0\right\}}^{\|\cdot\|_{L^{p}(\mathbb{R}^{d})}},\quad p\in(1,\infty),\\
&\mathcal{Q}=\mathrm{Id}-\mathcal{P}.
\end{align*}
In $\mathbb{R}^{d}$, $\mathcal{P}$ can be defined by $\mathcal{P}= \mathrm{Id}+(-\Delta)^{-1}\nabla {\rm{div}}$, or equivalently, $\mathcal{P}=(\mathcal{P}_{i j})_{1 \leqslant i, j \leqslant d}$, where $\mathcal{P}_{i j} \equiv \delta_{i j}+R_{i} R_{j}$ with $\delta_{i j}$ being the Kronecker delta ($\delta_{i j}=0$ for $i\neq j$ and $\delta_{i i}=0$) and $R_{i}$ being the Riesz transform with symbol $-\mathrm{i}\xi_1/|\xi|$. Obviously, $\mathcal{Q}= -(-\Delta)^{-1}\nabla {\rm{div}}$, and if $\div\, u=\div\, v=0$, it holds that
$
\mathcal{Q}(u\cdot\na v)= \mathcal{Q}(v\cdot\na u).
$
\end{itemize}
Next, we will recall some facts about the Littlewood-Paley decomposition and the nonhomogeneous Besov spaces (see \cite{B.C.D} for more details).
Choose a radial, non-negative, smooth function $\vartheta:\R^d\mapsto [0,1]$ such that
\begin{itemize}
  \item ${\rm{supp}} \;\vartheta\subset B(0, 4/3)$;
  \item $\vartheta(\xi)\equiv1$ for $|\xi|\leq3/4$.
\end{itemize}
Setting $\varphi(\xi):=\vartheta(\xi/2)-\vartheta(\xi)$, then we deduce that $\varphi$ has the following properties
\begin{itemize}
  \item ${\rm{supp}} \;\varphi\subset \left\{\xi\in \R^d: 3/4\leq|\xi|\leq8/3\right\}$;
  \item $\varphi(\xi)\equiv 1$ for $4/3\leq |\xi|\leq 3/2$;
  \item $\vartheta(\xi)+\sum_{j\geq0}\varphi(2^{-j}\xi)=1$ for any $\xi\in \R^d$;
  \item $\sum_{j\in \mathbb{Z}}\varphi(2^{-j}\xi)=1$ for any $\xi\in \R^d\setminus\{0\}$.
\end{itemize}
The nonhomogeneous and homogeneous dyadic blocks are defined as follows
\begin{equation*}
\begin{cases}
\forall\, u\in \mathcal{S'}(\R^d),\quad \Delta_ju=0,\; \text{if}\; j\leq-2;\quad
\Delta_{-1}u=\vartheta(D)u;\quad
\Delta_ju=\varphi(2^{-j}D)u,\; \; \text{if}\;j\geq0,\\
\forall\, u\in \mathcal{S'}\setminus \mathrm{P}(\R^d),\quad
\Delta_ju=\varphi(2^{-j}D)u,\; \; \text{if}\;j\in \mathbb{Z},
\end{cases}
\end{equation*}
where $\mathcal{S'}\setminus \mathrm{P}$ denote the tempered distribution modulo the polynomials and the pseudo-differential operator is defined by $\sigma(D):u\to\mathcal{F}^{-1}(\sigma \mathcal{F}u)$.

The Bony's decomposition reads as
\begin{align*}
&uv={T}_{u}v+{T}_{v}u+{R}(u,v)\quad\text{with}\quad
{T}_{u}v=\sum_{j\in\mathbb{Z}}{S}_{j-1}u{\Delta}_jv \quad\mbox{and} \quad{R}(u,v)=\sum_{j\in\mathbb{Z}}{\Delta}_ju\tilde{\Delta}_jv.
\end{align*}
We recall the definition of the Besov Spaces and norms.
\begin{definition}[Besov Space, See \cite{B.C.D}]
Let $s\in\mathbb{R}$ and $(p,r)\in[1, \infty]^2$. The nonhomogeneous and homogeneous Besov spaces are defined, respectively,
$$
B^{s}_{p,r}:=\f\{f\in \mathcal{S}':\;\|f\|_{B^{s}_{p,r}}:=\left\|2^{js}\|\Delta_jf\|_{L_x^p}\right\|_{\ell^r(j\geq-1)}<\infty\g\}
$$
and
$$
\dot{B}^{s}_{p,r}:=\f\{f\in \mathcal{S}'\setminus \mathrm{P}:\;\|f\|_{\dot{B}^{s}_{p,r}(\R^d)}:=\left\|2^{js}\|\dot{\Delta}_jf\|_{L_x^p}\right\|_{\ell^r(j\in \mathbb{Z})}<\infty\g\}.
$$
\end{definition}
For any $s>0$ and $(p,r)\in[1, \infty]^2$, we remark that
$$\|f\|_{B^{s}_{p,r}}\approx \|f\|_{L^{p}}+\|f\|_{\dot{B}^{s}_{p,r}}.$$
The following Bernstein's inequalities will be used in the sequel.
\begin{lemma}[See \cite{B.C.D}] \label{lem2.1} Let $\mathcal{B}$ be a ball and $\mathcal{C}$ be an annulus. There exists a constant $C>0$ such that for all $k\in \mathbb{N}\cup \{0\}$, any $\lambda\in \R^+$ and any function $f\in L^p$ with $1\leq p \leq q \leq \infty$, we have
\begin{align*}
&{\rm{supp}}\widehat{f}\subset \lambda \mathcal{B}\;\Rightarrow\; \|D^kf\|_{L^q}\leq C^{k+1}\lambda^{k+(\frac{1}{p}-\frac{1}{q})}\|f\|_{L^p},  \\
&{\rm{supp}}\widehat{f}\subset \lambda \mathcal{C}\;\Rightarrow\; C^{-k-1}\lambda^k\|f\|_{L^p} \leq \|D^kf\|_{L^p} \leq C^{k+1}\lambda^k\|f\|_{L^p}.
\end{align*}
\end{lemma}
Next we recall the following product law which will be used often in the sequel.
\begin{lemma}[See \cite{B.C.D}]\label{lp}
Assume $(s,p,r)$ satisfies \eqref{con1} and $\sigma>0$. Then
 there exists a constant $C$, depending only on $d,p,r,\sigma$ or $s$ such that
$$\|fg\|_{B^{\sigma}_{p,r}}\leq C\f(\|f\|_{L^\infty}\|g\|_{B^\sigma_{p,r}}+\|g\|_{L^\infty}\|f\|_{B^{\sigma}_{p,r}}\g),\quad \forall u,v\in L^\infty\cap B^{\sigma}_{p,r}.$$
Furthermore, due to the embedding $B^{s-1}_{p,r}(\R^d)\hookrightarrow L^{\infty}(\R^d)$, there holds
$$\|f\cdot \na g\|_{B^{s-1}_{p,r}}\leq C\|f\|_{B^{s-1}_{p,r}}\|g\|_{B^s_{p,r}},\quad \forall(f,g)\in B^{s-1}_{p,r}\times B^s_{p,r}$$
and
\begin{align}\label{cj}
\|f\cdot \na g\|_{B^{s}_{p,r}}\leq C\f(\|f\|_{B^{s-1}_{p,r}}\|g\|_{B^{s+1}_{p,r}}+\|f\|_{B^{s}_{p,r}}\|g\|_{B^{s}_{p,r}}\g),\quad \forall(f,g)\in B^{s}_{p,r}\times B^{s+1}_{p,r}.
\end{align}
\end{lemma}

\begin{lemma}\label{lem:P}
Assume $(s,p,r)$ satisfies \eqref{con1}. Then there exists a constant $C$, depending only on $d,p,r,s$, such that for all $u,v\in B^s_{p,r}$ with $\mathrm{div\,} u=\mathrm{div\,} v=0$,
\begin{align*}
&\|\mathcal{Q}(u\cdot \na v)\|_{B^s_{p,r}}\leq C \|u\|_{B^s_{p,r}}\|v\|_{B^s_{p,r}}
\end{align*}
and
\begin{align*}
&\|\mathcal{Q}(u\cdot \na v)\|_{B^{s-1}_{p,r}}\leq C \min\f\{\|u\|_{B^{s-1}_{p,r}}\|v\|_{B^s_{p,r}},\, \|v\|_{B^{s-1}_{p,r}}\|u\|_{B^s_{p,r}}\g\}.
\end{align*}
\end{lemma}

\begin{proof} Since the case $(s,p,r)=(1,\infty,1)$ has been proved in \cite{Pak} (see Proposition 8), we need only to consider the remaining case.

For the first part, it is easy to get
\begin{align*}
\|\mathcal{Q}(u\cdot \na v)\|_{B^s_{p,r}}\les  \|(\mathrm{Id}-\De_{-1})\mathcal{Q}(u\cdot \na v)\|_{\dot{B}^s_{p,r}}+\|\mathcal{Q}(u\cdot \na v)\|_{L^p}
=: \mathrm{I}+\mathrm{II}.
\end{align*}
Then, by Bony's decomposition, we have
\begin{align*}
\mathrm{I}&\lesssim  \sum^d_{i,j=1}\f\|(\mathrm{Id}-\De_{-1})[T_{\pa_j u^i}(\pa_iv^j)+ T_{\pa_i v^j}(\pa_ju^i)]\g\|_{\dot{B}^{s-1}_{p,r}}
+\sum^d_{i,j=1}\f\|(\mathrm{Id}-\De_{-1})R(u^i,\pa_iv^j)\g\|_{\dot{B}^{s}_{p,r}} \pr
\les \sum^d_{i,j=1}\f\|T_{\pa_j u^i}(\pa_iv^j)+ T_{\pa_i v^j}(\pa_ju^i)\g\|_{B^{s-1}_{p,r}}+\sum^d_{i,j=1}\f\|R(u^i,\pa_iv^j)\g\|_{B^{s}_{p,r}}
\pr \les\|\na u\|_{L^\infty}\|v\|_{B^s_{p,r}}+\|\na v\|_{L^\infty}\|u\|_{B^s_{p,r}}\les \|u\|_{B^s_{p,r}}\|v\|_{B^s_{p,r}},
\end{align*}
and due to the fact $\dot{B}^0_{p,1}\hookrightarrow L^p$
\begin{align*}
\mathrm{II}&\les \|\mathcal{Q}\Div(u\otimes v)\|_{\dot{B}^0_{p,1}}
\les \|u\otimes v\|_{\dot{B}^1_{p,1}} \les \|u\otimes v\|_{B^1_{p,1}} \les \|u\otimes v\|_{B^s_{p,r}}\les \|u\|_{B^s_{p,r}}\|v\|_{B^s_{p,r}}.
\end{align*}
For the second part, we just prove $\|\mathcal{Q}(u\cdot \na v)\|_{B^{s-1}_{p,r}}\leq C \|u\|_{B^{s-1}_{p,r}}\|v\|_{B^s_{p,r}}$ since it holds that $\mathcal{Q}(u\cdot \na v)=\mathcal{Q}(v\cdot \na u)$. Similarly,
\begin{align*}
\|\mathcal{Q}(u\cdot \na v)\|_{B^{s-1}_{p,r}}\les  \|(\mathrm{Id}-\De_{-1})\mathcal{Q}(u\cdot \na v)\|_{\dot{B}^{s-1}_{p,r}}+\|\mathcal{Q}(u\cdot \na v)\|_{L^p}=:\mathrm{III}+\mathrm{IV}.
\end{align*}
Then we get
\begin{align*}
\mathrm{III} \les  \|(\mathrm{Id}-\De_{-1})(u\cdot \na v)\|_{B^{s-1}_{p,r}}
\les  \|u\cdot \na v\|_{B^{s-1}_{p,r}}
\les \|u\|_{B^{s-1}_{p,r}}\|v\|_{B^{s}_{p,r}}.
\end{align*}
Notice that the support of these frequencies of $T_{u^j}\pa_j v^i$ and $T_{\pa_j v^i}u^j$ is away from zero and $\mathcal{Q}$ is a $S^0$-multiplier on their support, then we get
\begin{align*}
\mathrm{IV}\les  \|\mathcal{Q}(u\cdot \na v)\|_{B^{0}_{p,1}} &
\les \sum^d_{i,j=1}\f\|\na(-\De)^{-1}\pa_i(T_{u^j}\pa_j v^i+T_{\pa_j v^i}u^j)\g\|_{B^0_{p,1}}+\sum^d_{i,j=1}\f\|\na(-\De)^{-1}\pa_i\pa_jR(u^i,v^j)\g\|_{\dot{B}^0_{p,1}}
\pr\les \sum^d_{i,j=1}\f\|T_{u^j}\pa_j v^i+T_{\pa_j v^i}u^j\g\|_{B^0_{p,1}}+\sum^d_{i,j=1}\f\|R(u^i,v^j)\g\|_{\dot{B}^1_{p,1}}\pr
\les \sum^d_{i,j=1}\f\|T_{u^j}\pa_j v^i+T_{\pa_j v^i}u^j\g\|_{B^{s-1}_{p,r}}+\sum^d_{i,j=1}\f\|R(u^i,v^j)\g\|_{B^1_{p,1}}\pr
\les \|u\|_{B^{s-1}_{p,r}}\|\na v\|_{L^\infty}+\|u\|_{L^\infty}\|v\|_{B^{s}_{p,r}}+\|u\|_{L^\infty}\|v\|_{B^1_{p,1}}
\les \|u\|_{B^{s-1}_{p,r}}\|v\|_{B^{s}_{p,r}}.
\end{align*}
This completes the proof of Lemma \ref{lem:P}.
\end{proof}
Finally, we recall the regularity estimate for the transport-diffusion equation which reads
\begin{align}\label{eq:TDep}
\begin{cases}
\pa_t f+v\cdot \nabla f-\ep \Delta f=g,\\
f(0,x)=f_0(x),
\end{cases}
\end{align}
where $v:{\mathbb R}\times {\mathbb R}^d \to {\mathbb R}^d$, $f_0:{\mathbb R}^d\to {\mathbb R}^N$, and $g:{\mathbb R}\times {\mathbb R}^d\to {\mathbb R}^N$ are given.

\begin{lemma}[See \cite{B.C.D}]\label{cs}
Let $1\leq p,r\leq \infty$. Assume that
\begin{align*}
\sigma> -d \min\f(\frac{1}{p}, \frac{1}{p'}\g) \quad \mathrm{or}\quad \sigma> -1-d \min\f(\frac{1}{p}, \frac{1}{p'}\g)\quad \mathrm{if} \quad \mathrm{div\,} v=0.
\end{align*}
There exists a constant $C$, depending only on $d,p,r,\sigma$, such that for any smooth solution $f$ of \eqref{eq:TDep} and $t\geq 0$ we have
\begin{align}\label{ES2}
\|f\|_{L^\infty_tB^{\sigma}_{p,r}}\leq Ce^{CV_{p}(v,t)}\f(\|f_0\|_{B^\sigma_{p,r}}+\int^t_0\|g(\tau)\|_{B^{s}_{p,r}}\dd \tau\g),
\end{align}
with
\begin{align*}
V_{p}(v,t)=
\begin{cases}
\int_0^t \|\nabla v(s)\|_{B^{\frac{d}{p}}_{p,\infty}\cap L^\infty}\dd s,\quad \mathrm{if} \quad \sigma<1+\frac{d}{p},\\
\int_0^t \|\nabla v(s)\|_{B^{\sigma-1}_{p,r}}\dd s, \quad \quad \mathrm{if} \quad \sigma>1+\frac{d}{p}\ \mathrm{or}\ \{\sigma=1+\frac{d}{p} \mbox{ and } r=1\}.
\end{cases}
\end{align*}
If $f=v$, then for all $\sigma>0$ ($\sigma>-1$, if $\mathrm{div\,} v=0$), the estimate \eqref{ES2} holds with
$$V_{p}(t)=\int_0^t \|\nabla v(s)\|_{L^\infty}\dd s.$$
\end{lemma}

\section{Proof of Theorem \ref{th2}}\label{sec3}
In this section we prove Theorem \ref{th2}. Assume $(s,p,r)$ satisfies the conditions \eqref{con1}. For fixed $\ep>0$, by Theorem \ref{th1} we know that there exists $T_\ep=T(\|u_0\|_{B^s_{p,r}},\ep)>0$ such that (NS) has a unique solution $u$ in $\mathcal{C}([0,T_\ep];B^s_{p,r})$.
\subsection{Construction of Initial Data}\label{sec3.1}
We need to introduce smooth, radial cut-off functions to localize the frequency region. Let $\widehat{\phi}\in \mathcal{C}^\infty_0(\mathbb{R})$ be an even, real-valued and non-negative function on $\R$ and satisfy
\begin{numcases}{\widehat{\phi}(\xi)=}
1, &if\; $|\xi|\leq \frac{1}{4^d}$,\nonumber\\
0, &if\; $|\xi|\geq \frac{1}{2^d}$.\nonumber
\end{numcases}
\begin{lemma}\label{ley2}
For any $p\in[1,\infty]$, there exists two positive constants $c_1$ and $c_2$, we have
\begin{align*}
c_1\leq\|\phi\|_{L^p(\R)}\leq c_2.
\end{align*}
\end{lemma}
\begin{proof} For $p=\infty$, we have $\|\phi\|_{L^\infty}=\phi(0)>0$. In fact,
by the Fourier inverse formula and the Fubini theorem, we see that
$$\phi(0)=\frac{1}{2\pi}\int_{\R}\widehat{\phi}(\xi)\dd \xi>0$$
and
$$\|\phi\|_{L^\infty}=\sup_{x\in\R}\frac{1}{2\pi}\left|\int_{\R}\widehat{\phi}(\xi)\cos(x\xi)\dd \xi\right|\leq \frac{1}{2\pi}\int_{\R}\widehat{\phi}(\xi)\dd \xi.$$
For $p\in[1,\infty)$, there exists some $\delta>0$ such that for any $x\in B(0;\delta)$, we have $|\phi(x)|\geq \phi(0)/2$. Thus
\bbal
\|\phi\|_{L^p(\R)}\geq c_1.
\end{align*}
Since $\phi$ is a Schwartz function, we have for $100d\leq M\in \mathbb{Z}^+$
\bbal
|\phi(x)|\leq C(1+|x|)^{-M},
\end{align*}
which means that
\bbal
\|\phi\|_{L^p(\R)}\leq c_2.
\end{align*}
This completes the proof of Lemma \ref{ley2}.
\end{proof}
\begin{lemma}\label{ley3} Let $(s,p,r)$ satisfies \eqref{con1}. Define the initial data $u^n_0$ by
\begin{align*}
&u^n_0(x):=2^{-n(s+1)}
\f(
-\pa_2f_n,\;
\pa_1f_n,\;
0,\;
\cdots,\;
0
\g),\quad \text{where}\\
&f_n(x):=\phi(x_1)\cos \left(\frac{17}{12}2^nx_1\right)\prod_{i=2}^d\phi(x_i).
\end{align*}
Then for any $\sigma\in\R$, we have
$$\|u^n_0\|_{B^\sigma_{p,r}}\thickapprox 2^{n(\sigma-s)}\|\phi\|^d_{L^{p}}.$$
\end{lemma}
\begin{proof}
Straightforward computations yield
\bbal
\widehat{f_n}(\xi)=
\left[\widehat{\phi}\left(\xi_1-\frac{17}{12}2^n\right)+\widehat{\phi}\left(\xi_1+\frac{17}{12}2^n\right)\right]\prod_{i=2}^d\widehat{\phi}(\xi_i),
\end{align*}
which implies
\bal\label{s}
\mathrm{supp} \ \widehat{u^n_0}&\subset \left\{\xi\in\R^d: \ \frac{17}{12}2^n-\fr12\leq |\xi|\leq \frac{17}{12}2^n+\fr12\right\}\nonumber\\
&\subset \left\{\xi\in\R^d: \ \frac{4}{3}2^n\leq |\xi|\leq \frac{3}{2}2^n\right\}.
\end{align}
Recalling that $\varphi\equiv 1$ for $\frac43\leq |\xi|\leq \frac32$, we have
\begin{numcases}{\Delta_j(u^n_0)=}
u^n_0, &if $j=n$,\nonumber\\
0, &if $j\neq n$.\nonumber
\end{numcases}
Thus, we deduce that
\bbal
\|u^n_0\|_{B^\sigma_{p,r}(\R^d)}&=2^{n\sigma}\|u^n_0\|_{L^{p}(\R^d)}
= 2^{n(\sigma-s-1)}\f(\|\pa_1f_n\|_{L^{p}(\R^d)}+\|\pa_2f_n\|_{L^{p}(\R^d)}\g)
\approx2^{n(\sigma-s)}\|\phi\|^d_{L^{p}(\R)},
\end{align*}
which completes the proof of Lemma \ref{ley3}.
\end{proof}
\subsection{Key Approximation}\label{sec3.2}
From now on, we set $\ep_n:=2^{-2n}$.
\begin{proposition}\label{pr1} Let $u^n_0$ be given by Lemma \ref{ley3}. Assume that $(s,p,r)$ satisfies \eqref{con1}.
Then we have for $k\in\{-1,0,1\}$ and $t\in (0,1)$
\begin{align*}
\left\|e^{t\ep_n\Delta}u^n_0-u^n_0\right\|_{B^{s+k}_{p,r}}\approx t2^{kn}.
\end{align*}
\end{proposition}
\begin{proof}  By Taylor's formula, one has
$$\left(e^{t\ep_n\Delta}-\mathrm{Id}\right)u^n_0=\sum_{j\geq1}\frac{(t\ep_n\Delta)^j}{j!}u^n_0.$$
Using Bernstein's inequality and noticing that \eqref{s}, we deduce that
\begin{align*}
\left\|\sum_{j\geq1}\frac{(16t/9)^j}{j!}u^n_0\right\|_{B^{s+k}_{p,r}}
\leq\left\|\sum_{j\geq1}\frac{(t\ep_n\Delta)^j}{j!}u^n_0\right\|_{B^{s+k}_{p,r}}
\leq\left\|\sum_{j\geq1}\frac{(9t/4)^j}{j!}u^n_0\right\|_{B^{s+k}_{p,r}},
\end{align*}
It follows from Lemma \ref{ley3} that
\begin{align*}
\left(e^{16t/9}-1\right)2^{kn}
\leq\left\|\left(e^{t\ep_n\Delta}-\mathrm{Id}\right)u^n_0\right\|_{B^{s+k}_{p,r}}
\leq\left(e^{9t/4}-1\right)2^{kn},
\end{align*}
which is nothing but the desired result of Proposition \ref{pr1}.
\end{proof}
As a by-product of Proposition \ref{pr1}, we have the following
\begin{proposition}\label{pr2} Let $u^n_0$ be given by Lemma \ref{ley3}. Assume that $(s,p,r)$ satisfies \eqref{con1}.
Then we have for $k\in\{-1,0,1\}$ and $t\in (0,1)$
\begin{align}\label{y1}
\left\|e^{t\ep_n\Delta}u^n_0\cd\na e^{t\ep_n\Delta}u^n_0-u^n_0\cd\na u^n_0\right\|_{B^{s+k}_{p,r}}\leq C t
\end{align}
and \begin{align}\label{y2}
\left\|e^{t\ep_n\Delta}\left(u^n_0\cd\na u^n_0\right)-u^n_0\cd\na u^n_0\right\|_{B^{s}_{p,r}}
\leq  Ct.
\end{align}
\end{proposition}
\begin{proof} By the product law (see \eqref{cj} from Lemma \ref{lp}) and Proposition \ref{pr1}, we have
\bbal
\left\|e^{t\ep_n\Delta}u^n_0\cd\na e^{t\ep_n\Delta}u^n_0-u^n_0\cd\na u^n_0\right\|_{B^{s}_{p,r}}&\leq \left\|e^{t\ep_n\Delta}u^n_0-u^n_0\right\|_{B^{s}_{p,r}}\left\|e^{t\ep_n\Delta}u^n_0\right\|_{B^{s}_{p,r}}
+\left\|e^{t\ep_n\Delta}u^n_0-u^n_0\right\|_{B^{s-1}_{p,r}}\left\|e^{t\ep_n\Delta}u^n_0\right\|_{B^{s+1}_{p,r}}\\
&\quad+\|u^n_0\|_{B^{s}_{p,r}}\left\|e^{t\ep_n\Delta}u^n_0-u^n_0\right\|_{B^{s}_{p,r}}
+\|u^n_0\|_{B^{s-1}_{p,r}}\left\|e^{t\ep_n\Delta}u^n_0-u^n_0\right\|_{B^{s+1}_{p,r}}\\
&\leq Ct.
\end{align*}
Notice that $\mathrm{supp} \ \mathcal{F}\left(u^n_0\cd\na u^n_0\right)
\subset \left\{\xi\in\R^d: \ |\xi|\leq 3\cdot2^n\right\}$, then we have
\bbal
\left\|\left(e^{t\ep_n\Delta}-\mathrm{Id}\right)\left(u^n_0\cd\na u^n_0\right)\right\|_{B^{s}_{p,r}}
\leq \left(e^{9t}-1\right)\cdot\left\|u^n_0\cd\na u^n_0\right\|_{B^{s}_{p,r}}\leq Ct.
\end{align*}
This finishes the proof of Proposition \ref{pr2}.
\end{proof}
\subsection{Estimation of Errors}\label{sec3.3}
The following proposition involving the perturbation of solutions will play a crucial role in the proof of Theorem \ref{th2}.
\begin{proposition}\label{pr3}
Under the assumptions of Theorem \ref{th2}, then we have
\bal\label{YY1}
\left\|\mathbf{S}^0_{t}(u^n_0)-u^n_0+t\mathcal{P}\left(u^n_0\cd\na u^n_0\right)\right\|_{B^{s}_{p,r}}\leq Ct^{2}
\end{align}
and
\bal\label{YY2}
\left\|\mathbf{S}^{\ep_n}_{t}(u^n_0)-e^{t\ep_n\Delta}u^n_0+\int_0^te^{(t-\tau)\ep_n\Delta}\mathcal{P}\left(e^{\tau\ep_n\Delta}u^n_0\cd\na e^{\tau\ep_n\Delta}u^n_0\right)\dd\tau\right\|_{B^{s}_{p,r}}\leq Ct^{2}.
\end{align}
\end{proposition}
\begin{proof} We first prove \eqref{YY1}. By the Mean value theorem, one has
\bal\label{y}
\mathbf{S}^0_{t}(u^n_0)-u^n_0=-\int_0^t\mathcal{P}\left(\mathbf{S}^0_{\tau}(u^n_0)\cd\na \mathbf{S}^0_{\tau}(u^n_0)\right)\dd\tau.
\end{align}
By Minkowski's inequality and Lemmas \ref{lp} and \ref{lem:P}, we deduce that for $k\in\{-1,0,1\}$
\bal\label{z8}
\left\|\mathbf{S}^0_{t}(u^n_0)-u^n_0\right\|_{B^{s+k}_{p,r}}
\leq&~ \int_0^t\left\|\mathcal{P}\left(\mathbf{S}^0_{\tau}(u^n_0)\cd\na \mathbf{S}^0_{\tau}(u^n_0)\right)\right\|_{B^{s+k}_{p,r}}\dd\tau\nonumber\\
\leq&~ C\int_0^t\left\|\mathbf{S}^{0}_{\tau}(u^n_0)\right\|_{B^{s+k}_{p,r}} \|\mathbf{S}^{0}_{\tau}(u^n_0)\|_{B^{s}_{p,r}}+\left\|\mathbf{S}^{0}_{\tau}(u^n_0)\right\|_{B^{s-1}_{p,r}} \left\|\mathbf{S}^{0}_{\tau}(u^n_0)\right\|_{B^{s+k+1}_{p,r}}\dd\tau\nonumber\\
\leq&~ Ct2^{kn}.
\end{align}
From \eqref{y}, it follows that
\bal\label{l4}
\mathbf{S}^0_{t}(u^n_0)-u^n_0+t\mathcal{P}\left(u^n_0\cd\na u^n_0\right)=\int_0^t\mathcal{P}\left(u^n_0\cd\na u^n_0-\mathbf{S}^0_{\tau}(u^n_0)\cd\na \mathbf{S}^0_{\tau}(u^n_0)\right)\dd\tau.
\end{align}
Thus, we obtain from \eqref{z8} that
\bbal
&\left\|\mathbf{S}^0_{t}(u^n_0)-u^n_0+t\mathcal{P}\left(u^n_0\cd\na u^n_0\right)\right\|_{B^{s}_{p,r}}\\
\leq&~ \int_0^t\left\|\mathcal{P}\left(u^n_0\cd\na u^n_0-\mathbf{S}^0_{\tau}(u^n_0)\cd\na \mathbf{S}^0_{\tau}(u^n_0)\right)\right\|_{B^{s}_{p,r}}\dd\tau\\
\leq&~ C\int_0^t\left\|u^n_0-\mathbf{S}^0_{\tau}(u^n_0)\right\|_{B^{s}_{p,r}} \|u^n_0\|_{B^{s}_{p,r}}+\left\|u^n_0-\mathbf{S}^0_{\tau}(u^n_0)\right\|_{B^{s-1}_{p,r}} \|u^n_0\|_{B^{s+1}_{p,r}}\dd\tau\\
&+ C\int_0^t\left\|\mathbf{S}^0_{\tau}(u^n_0)\right\|_{B^{s}_{p,r}}\left\|u^n_0-\mathbf{S}^0_{\tau}(u^n_0)\right\|_{B^{s}_{p,r}}+ \left\|\mathbf{S}^0_{\tau}(u^n_0)\right\|_{B^{s-1}_{p,r}}\left\|u^n_0-\mathbf{S}^0_{\tau}(u^n_0)\right\|_{B^{s+1}_{p,r}}\dd\tau\\
\leq&~ Ct^2.
\end{align*}
We next prove \eqref{YY2}. By Duhamel's principle, one has
\bal\label{l}
\mathbf{S}^{\ep_n}_{t}(u^n_0)=e^{t\ep_n\Delta}u^n_0-\int_0^te^{(t-\tau)\ep_n\Delta}\mathcal{P}\left(\mathbf{S}^{\ep_n}_{\tau}(u^n_0)\cd\na \mathbf{S}^{\ep_n}_{\tau}(u^n_0)\right)\dd\tau,
\end{align}
which gives that for $k\in\{-1,0,1\}$
\bal\label{LL}
\left\|\mathbf{S}^{\ep_n}_{t}(u^n_0)-e^{t\ep_n\Delta}u^n_0\right\|_{B^{s+k}_{p,r}}
\leq&~ C\int_0^t\left\|\mathcal{P}\left(\mathbf{S}^{\ep_n}_{\tau}(u^n_0)\cd\na \mathbf{S}^{\ep_n}_{\tau}(u^n_0)\right)\right\|_{B^{s+k}_{p,r}}\dd\tau\nonumber\\
\leq&~ C\int_0^t\left\|\mathbf{S}^{\ep_n}_{\tau}(u^n_0)\right\|_{B^{s+k}_{p,r}} \left\|\mathbf{S}^{\ep_n}_{\tau}(u^n_0)\right\|_{B^{s}_{p,r}}+\left\|\mathbf{S}^{\ep_n}_{\tau}(u^n_0)\right\|_{B^{s-1}_{p,r}} \left\|\mathbf{S}^{\ep_n}_{\tau}(u^n_0)\right\|_{B^{s+k+1}_{p,r}}\dd\tau\nonumber\\
\leq&~ Ct2^{kn}.
\end{align}
From \eqref{l}, we have
\bbal
\mathbf{S}^{\ep_n}_{t}(u^n_0)-e^{t\ep_n\Delta}u^n_0&+\int_0^te^{(t-\tau)\ep_n\Delta}\mathcal{P}\left(e^{\tau\ep_n\Delta}u^n_0\cd\na e^{\tau\ep_n\Delta}u^n_0\right)\dd\tau\nonumber\\
&=-\int_0^te^{(t-\tau)\ep_n\Delta}\mathcal{P}\left(\mathbf{S}^{\ep_n}_{\tau}(u^n_0)\cd\na \mathbf{S}^{\ep_n}_{\tau}(u^n_0)-e^{\tau\ep_n\Delta}u^n_0\cd\na e^{\tau\ep_n\Delta}u^n_0\right)\dd\tau,
\end{align*}
from which and Minkowski's inequality, we obtain that
\bbal
&\left\|\mathbf{S}^{\ep_n}_{t}(u^n_0)-e^{t\ep_n\Delta}u^n_0+\int_0^te^{(t-\tau)\ep_n\Delta}\mathcal{P}\left(e^{\tau\ep_n\Delta}u^n_0\cd\na e^{\tau\ep_n\Delta}u^n_0\right)\dd\tau\right\|_{B^{s}_{p,r}}\\
\leq&~ C\int_0^t\left\|\mathcal{P}\left(\mathbf{S}^{\ep_n}_{\tau}(u^n_0)\cd\na \mathbf{S}^{\ep_n}_{\tau}(u^n_0)-e^{\tau\ep_n\Delta}u^n_0\cd\na e^{\tau\ep_n\Delta}u^n_0\right)\right\|_{B^{s}_{p,r}}\dd\tau\\
\leq&~ C\int_0^t\left\|\mathbf{S}^{\ep_n}_{\tau}(u^n_0)-e^{\tau\ep_n\Delta}u^n_0\right\|_{B^{s}_{p,r}} \|\mathbf{S}^{\ep_n}_{\tau}(u^n_0)\|_{B^{s}_{p,r}}+\left\|\mathbf{S}^{\ep_n}_{\tau}(u^n_0)-e^{\tau\ep_n\Delta}u^n_0\right\|_{B^{s-1}_{p,r}} \|\mathbf{S}^{\ep_n}_{\tau}(u^n_0)\|_{B^{s+1}_{p,r}}\dd\tau\\
&+ C\int_0^t\left\|e^{\tau\ep_n\Delta}u^n_0\right\|_{B^{s}_{p,r}}\left\|\mathbf{S}^{\ep_n}_{\tau}(u^n_0)-e^{\tau\ep_n\Delta}u^n_0\right\|_{B^{s}_{p,r}}+ \left\|e^{\tau\ep_n\Delta}u^n_0\right\|_{B^{s-1}_{p,r}}\left\|\mathbf{S}^{\ep_n}_{\tau}(u^n_0)-e^{\tau\ep_n\Delta}u^n_0\right\|_{B^{s+1}_{p,r}}\dd\tau.
\end{align*}
Using \eqref{LL} to the above, then we complete the proof of Proposition \ref{pr3}.
\end{proof}
{\bf Proof of Theorem \ref{th2}.}\;
From \eqref{U} and \eqref{V}, it follows that
\bbal
\mathbf{S}^{\ep_n}_{t}(u^n_0)-\mathbf{S}^{0}_{t}(u^n_0)&=\left(e^{t\ep_n\Delta}u^n_0-u^n_0\right)-\int_0^te^{(t-\tau)\ep_n\Delta}\mathcal{P}\left(e^{\tau\ep_n\Delta}u^n_0\cd\na e^{\tau\ep_n\Delta}u^n_0-u^n_0\cd\na u^n_0\right)\dd\tau\\
&\quad-\int_0^t\left(e^{(t-\tau)\ep_n\Delta}-\mathrm{Id}\right)\mathcal{P}\left(u^n_0\cd\na u^n_0\right)\dd\tau-\left[\mathbf{S}^0_{t}(u^n_0)-u^n_0+t\mathcal{P}\left(u^n_0\cd\na u^n_0\right)\right]\\
&\quad+\left[\mathbf{S}^{\ep_n}_{t}(u^n_0)-e^{t\ep_n\Delta}u^n_0+\int_0^te^{(t-\tau)\ep_n\Delta}\mathcal{P}\left(e^{\tau\ep_n\Delta}u^n_0\cd\na e^{\tau\ep_n\Delta}u^n_0\right)\dd\tau\right].
\end{align*}
Using Minkowski's inequality and Proposition \ref{pr2}, yields
\bbal
\left\|\int_0^te^{(t-\tau)\ep_n\Delta}\mathcal{P}\left(e^{\tau\ep_n\Delta}u^n_0\cd\na e^{\tau\ep_n\Delta}u^n_0-u^n_0\cd\na u^n_0\right)\dd\tau\right\|_{B^{s}_{p,r}}
\leq Ct^2
\end{align*}
and
\bbal
\left\|\int_0^t\left(e^{(t-\tau)\ep_n\Delta}-\mathrm{Id}\right)\mathcal{P}\left(u^n_0\cd\na u^n_0\right)\dd\tau\right\|_{B^{s}_{p,r}}
\leq Ct^2.
\end{align*}
Thus, from Proposition \ref{pr1}-\ref{pr3} we deduce that for samll time $t\in(0,T_0]$
\bbal
\f\|\mathbf{S}^{\ep_n}_{t}(u^n_0)-\mathbf{S}^{0}_{t}(u^n_0)\g\|_{B^s_{p,r}}&\geq \left\|e^{t\ep_n\Delta}u^n_0-u^n_0\right\|_{B^s_{p,r}}-Ct^2\geq ct-Ct^2\geq c_0t.
\end{align*}
This completes the proof of Theorem \ref{th2}. {\hfill $\square$}
\section{Proof of Theorem \ref{th3}}\label{sec4}

We modify the initial data $u^{n,k}_0$ by
\begin{align*}
u^{n,k}_0(x):=2^{-n(s+1)}
\f(
-\pa_2f_{n}(x-k\mathbf{e}),\;
\pa_1f_{n}(x-k\mathbf{e}),\;
0,\;
\cdots,\;
0
\g)\quad \text{with}\quad
\mathbf{e}:=(1,0,\cdots,0),
\end{align*}
where $f_{n}$ be given in Lemma \ref{ley3} and $k=k(n)$ can be chosen later.

\begin{proposition}\label{pro5-2}
Under the assumptions of Theorem \ref{th3}, we have for any $\ep_n\geq0$
\bbal
\f\|\mathbf{S}^{\ep_n}_{t}(S_n\psi+u^{n,k}_0)-\mathbf{S}^{\ep_n}_{t}(S_n\psi)
-\mathbf{S}^{\ep_n}_{t}(u^{n,k}_0)\g\|_{B^{s}_{p,r}}\leq C2^{\fr{n}2}\f(\int_0^t\f\|\mathbf{S}^{\ep_n}_{\tau}(u^{n,k}_0)\g\|_{B^{s}_{p,r}}\dd\tau\g)^{\fr12}.
\end{align*}
\end{proposition}
\begin{proof} Theorem \ref{th1} tells us that $\mathbf{S}^{\ep_n}_{t}(S_n\psi+u^{n,k}_0),\mathbf{S}^{\ep_n}_{t}(S_n\psi),\mathbf{S}^{\ep_n}_{t}(u^{n,k}_0)\in \mathcal{C}([0,T];B^s_{p,r})$ and has common lifespan $T\thickapprox1$. Moreover, there holds
\bbal
\f\|\mathbf{S}^{\ep_n}_{t}(S_n\psi+u^{n,k}_0),\mathbf{S}^{\ep_n}_{t}(S_n\psi),\mathbf{S}^{\ep_n}_{t}(u^{n,k}_0)\g\|_{L^\infty_TB^s_{p,r}}\leq C.
\end{align*}
We set $\mathbf{w}:=\mathbf{S}^{\ep_n}_{t}(S_n\psi+u^{n,k}_0)-\mathbf{S}^{\ep_n}_{t}(S_n\psi)-\mathbf{S}^{\ep_n}_{t}(u^{n,k}_0)$. It is obvious that $\mathbf{w}$ solves
\begin{equation}\label{z}
\begin{cases}
\partial_t\mathbf{w}-\ep_n\Delta \mathbf{w}+\mathbf{S}^{\ep_n}_{t}(S_n\psi+u^{n,k}_0)\cdot\nabla\mathbf{w}
=\mathcal{Q}\f(\mathbf{w}\cdot\nabla\mathbf{S}^{\ep_n}_{t}(S_n\psi+u^{n,k}_0)\g)\\
~~~~~~
-\mathcal{P}\f(\mathbf{w}\cdot\nabla\big(\mathbf{S}^{\ep_n}_{t}(S_n\psi)+\mathbf{S}^{\ep_n}_{t}(u^{n,k}_0)\big)
+\mathbf{S}^{\ep_n}_{\tau}(S_n\psi)\cdot\nabla\mathbf{S}^{\ep_n}_{\tau}(u^{n,k}_0)
+\mathbf{S}^{\ep_n}_{\tau}(u^{n,k}_0)\cdot\nabla\mathbf{S}^{\ep_n}_{\tau}(S_n\psi)\g),\\
\mathbf{w}(0,x)=0.
\end{cases}
\end{equation}
By the interpolation inequality, we obtain
\bal\label{lyz}
\|\mathbf{w}\|_{B^s_{p,r}}\leq C\|\mathbf{w}\|^{\fr12}_{B^{s-1}_{p,r}}\|\mathbf{w}\|^{\fr12}_{B^{s+1}_{p,r}}
\leq C2^{\fr{n}2}\|\mathbf{w}\|^{\fr12}_{B^{s-1}_{p,r}}.
\end{align}
Using the regularity estimate of transport-diffusion equation from Lemma \ref{cs}, and then Lemmas \ref{lp} and \ref{lem:P}, we obtain that
\bbal
\f\|\mathbf{w}\g\|_{B^{s-1}_{p,r}}&\leq C\int_0^t\f\|\mathbf{w}\g\|_{B^{s-1}_{p,r}}\dd\tau+C\int_0^t\f\|\mathbf{S}^{\ep_n}_{\tau}(u^{n,k}_0)\g\|_{B^{s}_{p,r}}\dd\tau.
\end{align*}
Combining Gr\"{o}nwall's  inequality and  \eqref{lyz} completes the proof of Proposition \ref{pro5-2}.
\end{proof}
\begin{proposition}\label{pro5-1}
Under the assumptions of Theorem \ref{th3}, we have for any $\ep_n\geq0$
\bbal
\f\|\mathbf{S}^{\ep_n}_{t}(\psi+u^{n,k}_0)-\mathbf{S}^{\ep_n}_{t}(S_n\psi+u^{n,k}_0)\g\|_{B^{s}_{p,r}}\leq C\f\|(\mathrm{Id}-S_n)\psi\g\|_{B^{s}_{p,r}}.
\end{align*}
\end{proposition}
\begin{proof}  Theorem \ref{th1} tells us that $\mathbf{S}^{\ep_n}_{t}(\psi+u^{n,k}_0),\mathbf{S}^{\ep_n}_{t}(S_n\psi+u^{n,k}_0)\in \mathcal{C}([0,T];B^s_{p,r})$ and has common lifespan $T\thickapprox1$. Moreover, there holds
\bbal
\f\|\mathbf{S}^{\ep_n}_{t}(\psi+u^{n,k}_0),\mathbf{S}^{\ep_n}_{t}(S_n\psi+u^{n,k}_0)\g\|_{L^\infty_TB^s_{p,r}}\leq C.
\end{align*}
We set $\mathbf{v}:=\mathbf{S}^{\ep_n}_{t}(\psi+u^{n,k}_0)-\mathbf{S}^{\ep_n}_{t}(S_n\psi+u^{n,k}_0)$.
It is obvious that $\mathbf{v}$ solves
\begin{equation*}
\begin{cases}
\partial_t\mathbf{v}-\ep_n\Delta \mathbf{v}+\mathbf{S}^{\ep_n}_{t}(\psi+u^{n,k}_0)\cdot\nabla\mathbf{v}
=\mathcal{Q}\f(\mathbf{S}^{\ep_n}_{t}(\psi+u^{n,k}_0)\cdot\nabla\mathbf{v}\g)
-\mathcal{P}\f(\mathbf{v}\cdot\nabla\mathbf{S}^{\ep_n}_{t}(S_n\psi+u^{n,k}_0)\g),\\
\mathbf{v}(0,x)=(\mathrm{Id}-S_n)\psi.
\end{cases}
\end{equation*}
By Lemma \ref{cs}, and then Lemma \ref{lem:P}, we obtain that
\bbal
\f\|\mathbf{v}\g\|_{B^{s-1}_{p,r}}&\leq \|(\mathrm{Id}-S_n)\psi\|_{B^{s-1}_{p,r}}+C\int_0^t\f\|\mathbf{v}\g\|_{B^{s-1}_{p,r}}\dd\tau
\end{align*}
and
\bbal
\f\|\mathbf{v}\g\|_{B^{s}_{p,r}}
\leq&~ \f\|(\mathrm{Id}-S_n)\psi\g\|_{B^{s}_{p,r}}
+C2^n\int^t_0\|\mathbf{v}\|_{B^{s-1}_{p,r}}\dd \tau+C\int^t_0\|\mathbf{v}\|_{B^{s}_{p,r}}\dd \tau,
\end{align*}
which gives the desired result of Proposition \ref{pro5-1}.
\end{proof}

Now, we prove Theorem \ref{th3} by dividing it into three steps.

{\bf Step 1.}\; By Theorem \ref{th2}, we have for a short time $t\in[0,T_0]$
\bbal
\f\|\mathbf{S}^{\ep_n}_{t}(u^{n,k}_{0})-\mathbf{S}^{0}_{t}(u^{n,k}_{0})\g\|_{B^s_{p,r}}\geq c_0t>0.
\end{align*}

{\bf Step 2.}\;
Notice that $\mathbf{S}^{\ep_n}_{t}(u^{n,k}_0)=\mathbf{S}^{\ep_n}_{t}(u^{n,0}_0)(t,x-k\mathbf{e})$. For fixed $n$ and any $(t,x)$, thanks to the decay, we have
$\lim_{|k|\rightarrow \infty} \mathbf{S}^{\ep_n}_{t}(u^{n,k}_0)=0.
$
By the Lebesgue Dominated Convergence Theorem, we deduce that
\bbal
&\lim_{|k|\rightarrow \infty} \int_0^{T_0}\f\|\mathbf{S}^{\ep_n}_{\tau}(u^{n,k}_0)\g\|_{B^{s}_{p,r}}\dd \tau=0.
\end{align*}
By Proposition \ref{pro5-2}, for fixed $n$ and $t\in[0,T_0]$, one can find $k$ with $|k|$ sufficiently large such that
\bbal
\f\|\mathbf{S}^{\ep_n}_{t}(S_n\psi+u^{n,k}_0)-\mathbf{S}^{\ep_n}_{t}(S_n\psi)
-\mathbf{S}^{\ep_n}_{t}(u^{n,k}_0)\g\|_{B^{s}_{p,r}}\leq \fr{1}{n}.
\end{align*}

{\bf Step 3.}\; We decompose the difference of $\mathbf{S}^{\ep_n}_{t}(\psi+u^{n,k}_{0})$ and $\mathbf{S}^{0}_{t}(\psi+u^{n,k}_{0})$ as follows
\bbal
\mathbf{S}^{\ep_n}_{t}(&\psi+u^{n,k}_{0})-\mathbf{S}^{0}_{t}(\psi+u^{n,k}_{0})
=\mathbf{S}^{\ep_n}_{t}(u^{n,k}_{0})-\mathbf{S}^{0}_{t}(u^{n,k}_{0})+\mathbf{S}^{\ep_n}_{t}(S_n\psi)-\mathbf{S}^{0}_{t}(S_n\psi)\\
&\quad+\left[\underbrace{\mathbf{S}^{\ep_n}_{t}(\psi+u^{n,k}_{0})-\mathbf{S}^{\ep_n}_{t}(S_n\psi+u^{n,k}_{0})}_{:=\mathbf{I}_1}
+\underbrace{\mathbf{S}^{\ep_n}_{t}(S_n\psi+u^{n,k}_{0})-\mathbf{S}^{\ep_n}_{t}(S_n\psi)
-\mathbf{S}^{\ep_n}_{t}(u^{n,k}_{0})}_{:=\mathbf{I}_2}\right]\\
&\quad-\left[\underbrace{\mathbf{S}^{0}_{t}(\psi+u^{n,k}_{0})-\mathbf{S}^{0}_{t}(S_n\psi+u^{n,k}_{0})}_{:=\mathbf{I}_3}
+\underbrace{\mathbf{S}^{0}_{t}(S_n\psi+u^{n,k}_{0})-\mathbf{S}^{0}_{t}(S_n\psi)
-\mathbf{S}^{0}_{t}(u^{n,k}_{0})}_{:=\mathbf{I}_4}\right]
\end{align*}
and
\bbal\mathbf{S}^{\ep_n}_{t}(S_n\psi)-\mathbf{S}^{0}_{t}(S_n\psi)
=\mathbf{S}^{\ep_n}_{t}(\psi)-\mathbf{S}^{0}_{t}(\psi)
+\underbrace{\mathbf{S}^{\ep_n}_{t}(S_n\psi)-\mathbf{S}^{\ep_n}_{t}(\psi)}_{:=\mathbf{I}_5}
+\underbrace{\mathbf{S}^{0}_{t}(\psi)-\mathbf{S}^{0}_{t}(S_n\psi)}_{:=\mathbf{I}_6}.
\end{align*}
Hence, from Proposition \ref{pro5-1} and {\bf Step 2} we deduce
\bbal
\f\|\mathbf{S}^{\ep_n}_{t}(\psi+u^{n,k}_{0})-\mathbf{S}^{0}_{t}(\psi+u^{n,k}_{0})\g\|_{B^s_{p,r}}
&\geq \f\|\mathbf{S}^{\ep_n}_{t}(u^{n,k}_{0})-\mathbf{S}^{0}_{t}(u^{n,k}_{0})\g\|_{B^s_{p,r}}
-\f\|\mathbf{S}^{\ep_n}_{t}(\psi)-\mathbf{S}^{0}_{t}(\psi)\g\|_{B^s_{p,r}}-\sum_{i=1}^6\|\mathbf{I}_i\|_{B^s_{p,r}}\\
&\geq \f\|\mathbf{S}^{\ep_n}_{t}(u^{n,k}_{0})-\mathbf{S}^{0}_{t}(u^{n,k}_{0})\g\|_{B^s_{p,r}}-\f\|\mathbf{S}^{\ep_n}_{t}(\psi)-\mathbf{S}^{0}_{t}(\psi)\g\|_{B^s_{p,r}}
\\& \quad-C\f\|(\mathrm{Id}-S_n)\psi\g\|_{B^{s}_{p,r}}-\fr{1}{n},
\end{align*}
which follows from {\bf Step 1} and Theorem \ref{th1} that for $t$ small enough
\bbal
\liminf_{n\rightarrow \infty}\f\|\mathbf{S}^{\ep_n}_{t}(\psi+u^{n,k}_{0})-\mathbf{S}^{0}_{t}(\psi+u^{n,k}_{0})\g\|_{B^s_{p,r}}\geq \fr{c_0t}{2}.
\end{align*}
This completes the proof of Theorem \ref{th3}. {\hfill $\square$}
\section*{Acknowledgements}
J. Li is supported by the National Natural Science Foundation of China (11801090 and 12161004) and Jiangxi Provincial Natural Science Foundation (20212BAB211004 and 20224BAB201008). Y. Yu is supported by the National Natural Science Foundation of China (12101011). W. Zhu is supported by the National Natural Science Foundation of China (12201118) and Guangdong
Basic and Applied Basic Research Foundation (2021A1515111018).

\section*{Declarations}
\noindent\textbf{Data Availability} No data was used for the research described in the article.

\vspace*{1em}
\noindent\textbf{Conflict of interest}
The authors declare that they have no conflict of interest.

\end{document}